\documentclass[12pt]{article}
\usepackage{amsmath,amsthm, amssymb}
\usepackage{graphicx}
\usepackage{enumitem}
\usepackage{mathrsfs}

\newtheorem{theorem}{Theorem}[section]
\newtheorem*{theoremn}{Theorem}

\newtheorem{proposition}{Proposition}[section]

\newtheorem{lemma}{Lemma}[section]

\newtheorem{definition}{Definition}[section]

\def\ie{{\em i.e.,\ }}
\def\eg{{\em e.g.\ }}

\def\SS{{\mathcal S}} 
\def\shape{\epsilon}
\newfont\bbf{msbm10 at 12pt}

\def\eps{\varepsilon}
\def\R{{\mathbb R}}

\def\P{{\mathcal P}}

\def\T{{\mathcal T}}

\def\htop{h_{top}}
\def\orb{\mbox{\rm orb}}

\begin{document}
\bibliographystyle{plain}
\title{On the structure of isentropes of polynomial maps}
\author{Henk Bruin  and Sebastian van Strien}
\date{Version of 26 February}
\maketitle
\begin{abstract}
The structure of isentropes (\ie level sets of constant topological entropy)
including the monotonicity of entropy, has been studied for polynomial interval maps
since the 1980s.
We show that isentropes of multimodal polynomial families need not be 
locally connected and that entropy does in general not depend monotonically
on a single critical value.
\end{abstract}

\section{Introduction and Statement of Results.}
Topological entropy $\htop$ was introduced by Adler et al.\ \cite{AKM}, 
and the for many purposes more practical approach using $\eps$-$n$-separated/spanning sets was developped by Bowen \cite{bowen}
and Dinaburg \cite{dinaburg}.
For interval maps $f:I \to I$, Misiurewicz \& Szlenk \cite{MSz}
showed that $\htop(f)$ represents the exponential growth rate of 
the number of periodic points\footnote{where it is assumed that for every $n$, every separate branch of $f^n$ contains a 
bounded number of $n$-periodic points. That this assumption is necessary
emerges \eg from \cite{KK}.}, the number of laps or the variation:
\begin{eqnarray}\label{eq:htop}
\htop(f) &=& \lim_{n\to\infty} \frac1n \log \# \{ x : f^n(x) = x \} \nonumber \\
&=& \lim_{n\to\infty} \frac1n \log \ell(f^n) \\
&=& \lim_{n\to\infty} \frac1n \log \mbox{Var}(f^n). \nonumber
\end{eqnarray}
Here the {\em lapnumber} $\ell(f^n)$ denotes the number of maximal intervals $J$ such that $f^n|_J$ is monotone,
and the {\em variation} is 
$$
\mbox{Var}(f) = \sup_k \sup_{x_0 < \dots < x_k} \sum_{j=1}^k |f(x_j)-f(x_{j-1})|.
$$

The question of how $\htop(f)$ depends on the map $f$ is a major theme within interval dynamics since 
numerical observations made  in the 1970's. The simplest question of this type is whether 
for the logistic family $f_a(x) = ax(1-x)$, the entropy $\htop(f_a)$ is increasing in $a$. 
 In the early 80's, this was proved by  Douady \& Hubbard \cite{Dou,DH}, Milnor \& Thurston  \cite{MT} and Sullivan, see  \cite{MS}. 
Another elegant proof was given by Tsujii \cite{Tsu}.
It is worth emphasizing that all known monotonicity proofs rely in some way on complex analysis, and so does the recent proof of van Strien \& Rempe for the sine family \cite{RvS}.
Milnor \& Tresser \cite{MTr} (see also \cite{Mil}) proved that entropy is also monotone for the family of {\em cubic} polynomials, and Radulescu \cite{Radu} proved the same thing for a certain two-dimensional {\em slice of quartic} polynomials.
In this case, the parameter space is two-dimensional, and then (as for all higher dimensional spaces)
monotonicity of entropy means that the {\em isentropes} (\ie sets of constant entropy) are connected.

For polynomials of {\em arbitrary degree} $d=b+1$ (\ie $b$-modal polynomials) monotonicity of entropy was proved recently in \cite{BS}.
More precisely, for fixed $\shape \in \{-1 , 1\}$ and define 
$$
\P^b_\shape = \left\{ f:[-1, 1] \to [-1,1] : 
\begin{array}{l}
f  \text{ is a polynomial of }\deg(f) = b+1, \\[1mm] 
f(-1) = \shape,\ f(1) \in \{ -1, 1\} \text{ and } f \text{ has } \\[1mm]
 b \text{ distinct critical points in } (-1,1) 
\end{array} \right\}.
$$

\begin{theoremn}[Milnor's Monotonicity Conjecture, see  \cite{BS}]
For each $h\ge 0$, the isentrope
$$
L_h = \{ f \in \P^b_\shape : \htop(f) = h \}
$$
is connected.
\end{theoremn}
The main purpose of this paper is to
present two theorems which show that isentropes of $\P^b_\shape$ are nevertheless complicated sets:

\begin{theorem}[Non-monotonicity w.r.t.\ natural parameters]\label{thm:non-mono}
Let $f_v \in \P^b_\shape$ denote the polynomial map with critical values $v = (v_1, \dots, v_b)$.
For $b \ge 2$, there are fixed values of $v_2, \dots , v_b$ such that the map
$$
v_1 \mapsto \htop(f_v)
$$
is not monotone.
\end{theorem}

\begin{theorem}[Non-local connectivity]\label{thm:isentropeNLC}
For any $b \ge 4$, there is a dense set $H\subset  [0, \log(b-1)]$ such that for each
$h\in H$,  the isentrope $L_h$ of $\P^b_\shape$ is not locally connected.
\end{theorem}

The proof of these theorems will be given in  Sections~\ref{sec:nonmono} and \ref{sec:locally_disconnected}. 

The method used in the proof of the last  theorem is related to a result by Friedman \& Tresser \cite{FrTr} which states
that the zero isentrope of bimodal circle maps is not locally connected. 
This is one aspect of a description of the boundary of chaos behavior for bimodal circle maps;
for a discussion of a more extensive program, see \eg \cite{GLT,MTr}.

{\bf Questions:} For which values of entropy are the isentropes not locally connected? 
The values $h = \log s$ for which our proof works are for algebraic $s$. Are there non-algebraic values of $s$ such that the isentrope of level $\log s$ not  locally connected? Is it possible to construct non-local connectivity when 
$h\in (\log(b-1),\log(b+1))$.  Are there values $h$ so that there 
exists a dense subset of $g \in L_h$ so that $L_h$ is non-locally 
connected at $g$? 

For other questions, see \cite{vS2}.

\subsection{Some questions about other multimodal interval maps}

Before giving the proof of the main theorems of this paper, let us turn to other models of $b$-modal families. We will consider two such families,
namely the tent maps and the stunted sawtooth maps.
We will always scale them such that $f:[-1,1] \to [-1,1]$, $f(-1) = \shape \in \{-1,1\}$, $f(1) \in \{-1,1\}$,
and there are $b$ critical point $c_i = -1+2i/(b+1)$, $i=0,\dots,b+1$ so by default $c_0 = -1$ and $c_{b+1}=1$
are the boundary points.
The critical values $v_i = f(c_i) \in [-1,1]$ with $v_0 = \shape$, 
$v_{b+1} = (-1)^{b+1} \shape$ will satisfy
\begin{equation}\label{eq:v}
(v_i-v_{i-1}) \cdot \shape \ \left\{ \begin{array}{ll}  
< 0 & \text{ if } i \text{ is odd;}\\
> 0 & \text{ if } i \text{ is even.}
\end{array} \right.
\end{equation}
\begin{itemize}[topsep=-0.15cm,itemsep=0.3ex,leftmargin=0.6cm]
\item A $b$-modal tent-map $f$ is obtained by choosing $b$ critical values $v_i$ satisfying
\eqref{eq:v}, and then defining
$$
f(x) =  \left\{ \begin{array}{ll}  
v_i & \text{ if } x = c_i;\\
\text{by linear interpolation } & \text{ if } x \notin \{ c_0, \dots, c_{b+1} \}. 
\end{array} \right.
$$
The class of $b$-modal tent-maps is denoted by $\T^b_\shape$.
If $f \in \T^b_\shape$ has constant slope $f' = \pm s$, then $\htop(f) = \max\{ \log s, 0\}$.
This follows immediately from the variation interpretation of entropy in \eqref{eq:htop}.
By the same token, $|f'| \ge s$ implies $\htop(f) \ge \log s$.
A well-known result by Milnor \& Thurston \cite{MT} (preceded by Parry \cite{Parry} 
in the context of $\beta$-transformations) is that every $b$-modal interval map
$g$ of entropy $\htop(g) = \log s > 0$ is semi-conjugate to some
$f \in \T^{b'}_\shape$ with constant slope $\pm s$ for some $b' \le b$.
However, the dynamics of $f$ may be strictly less complicated, because
maps in $\T^b_\shape$ allow only periodic intervals of limited types. Hence the semi-conjugacy
connecting $f \in \T^b_\shape$ and $g \in \P^b_\shape$ can collapse (pre-)periodic intervals of $g$ to (pre-)periodic
points of $f$.
\item A $b$-modal stunted sawtooth map is obtained from a fixed $b$-modal
sawtooth maps as follows.
First constructed an unstunted sawtooth map (or full $b$-modal tent-map) 
$S_0:[-1,1] \to [-1,1]$ by setting $S_0(c_i) = \shape$ or $-\shape$ if $i$ is odd or even 
respectively, and defining $S_0$ in between by linear interpolation.
Then choose new critical values $v_i$ satisfying \eqref{eq:v},
and find (for $i = 1, \dots, b$) the maximal closed intervals $Z_i \owns c_i$ such that
$(S_0(x)-v_i) \cdot (-1)^i \le 0$ for all $x \in Z_i$.
Finally, define the  {\em stunted sawtooth map} as
$$
f(x) = \left\{ \begin{array}{ll}
v_i & \text{ if } x \in Z_i \text{ for some } i \in \{ 1, \dots, b\}; \\
S_0(x) & \text{ otherwise. }
\end{array} \right.
$$
The class of $b$-modal stunted sawtooth maps is denoted by $\SS^b_\shape$.
This class of maps was heavily used in the proof of Milnor's conjecture, see \cite{BS}. 
\end{itemize}
It is again natural to parametrize maps in $\T^b_\shape$ and $\SS^b_\shape$ by their critical values $v = (v_1, \dots, v_b)$.

\begin{figure}
\begin{center}
\unitlength=1.3mm
\begin{picture}(12,25)(17,3)
\put(-35,0){\resizebox{4cm}{4cm}{\includegraphics{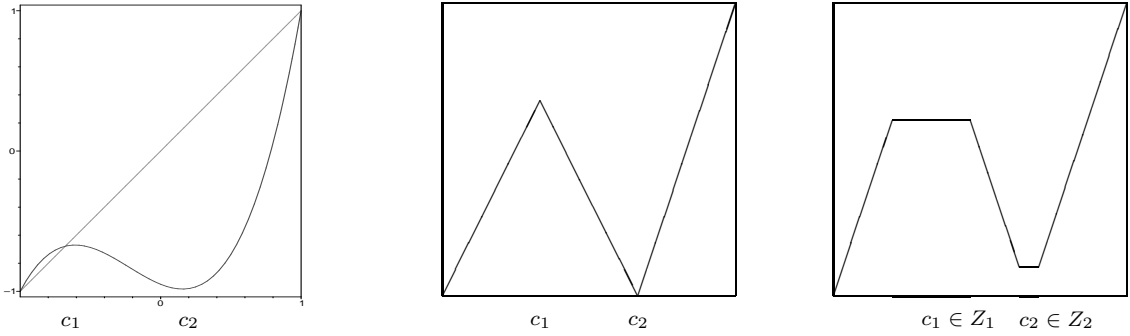}}}
\put(-29,-2){\scriptsize $c_1$} 
\put(-17,-2){\scriptsize $c_2$} 
\thinlines
\put(10,1){\line(1,0){30}}
\put(10,1){\line(0,1){30}}
\put(40,31){\line(-1,0){30}}
\put(40,31){\line(0,-1){30}}
\put(10,1){\line(1,2){10}}
\put(20,21){\line(1,-2){10}}
\put(30,1){\line(1,3){10}}
\put(19,-2){\scriptsize $c_1$} 
\put(29,-2){\scriptsize $c_2$} 
\put(50,1){\line(1,0){30}}
\put(50,1){\line(0,1){30}}
\put(80,31){\line(-1,0){30}}
\put(80,31){\line(0,-1){30}}
\put(50,1){\line(1,3){6}}\put(56,19){\line(1,0){8}}
\put(64,19){\line(1,-3){5}}\put(69,4){\line(1,0){2}}
\put(71,4){\line(1,3){9}}
\put(59,-2){\scriptsize $c_1 \in Z_1$}\put(56,0.95){\line(1,0){8}} 
\put(69,-2){\scriptsize $c_2 \in Z_2$}\put(69,0.95){\line(1,0){2}} 
\end{picture}
\end{center}
\caption{\label{fig:bimodals} Bimodal maps with $\shape = -1$: polynomial, tent and 
stunted sawtooth.}
\end{figure}

\begin{proposition}\label{prop:isentrope_sawtooth}
Let $f_v \in \SS^b_\shape$ denote the stunted sawtooth map with critical values  $v = (v_1, \dots, v_b)$.
For each $i \in \{1, \dots, b\}$, the map
$$
v_i \mapsto \htop(f_v)
$$
is increasing/decreasing if $i$ is odd/even.
All isentropes of $\SS^b_\shape$ are connected and locally connected.
\end{proposition}
\begin{proof} See Section~\ref{sec:nonmono}. 
\end{proof}

Surprisingly, the corresponding questions seem to be unanswered for the space $\T^n_\shape$:

{\bf Question:} Are the isentropes of $\T^b_\shape$ connected? Are they all locally connected?

\section{Non-monotonicity w.r.t.\ single critical values}
\label{sec:nonmono}

Let $\zeta_i = \shape \cdot (-1)^i \cdot v_i$, so increasing $\zeta_i$ always means decreasing the width of the plateau $Z_i$.
Using these coordinates, we can reformulate the first statement of
Proposition~\ref{prop:isentrope_sawtooth} as
\begin{quote}
The map $\zeta = (\zeta_1, \dots, \zeta_d) \to \htop(T_\zeta)$ 
is non-decreasing in each coordinate.
\end{quote}

\begin{proof}[Proof of Proposition~\ref{prop:isentrope_sawtooth}]
Increasing a parameter $\zeta_i$ makes a plateau narrower, and affects none of
the orbits that never enter $Z_i$. Therefore only new orbits are created, and none destroyed. Hence entropy is non-decreasing in each $\zeta_i$.

Now to prove local connectedness of isentrope $L_s$, let $x, y \in L_s$ be two arbitrary points; they subtend a parallelepiped $P$ of dimension $\le b$
with faces perpendicular to the coordinate axes.
That is, $x$ and $y$  are vertices of $P$, and we denote the vertices
at which all parameters $\zeta_i$ are minimal and maximal by $l$ and $u$
respectively.
Foliate $P$ by one-dimensional strands $\gamma$, connecting $l$ to $u$,
such that all parameters vary monotonically along $\gamma$.
Then for each such $\gamma$, there is a unique compact arc (or singleton)
$J_\gamma$ at which the entropy is constant $s$.
The union $\cup_\gamma J_\gamma = P \cap L_s$ and due to continuity of entropy,
it is connected.
Hence between any two $x, y \in L_s$ there is an arc connecting 
$x$ and $y$ within $L_s$ with diameter $\le d(x,y)$ in sum-metric $d$.
This proves local connectivity.
\end{proof}

The analogous monotonicity statement is false for $\P^b$ for $b \ge 2$.
The below proof demonstrates that 
parametrizing the cubic family by its critical values $v_1, v_2$
results in non-monotonicity in (each of) its parameters separately.

\begin{proof}[Proof of Theorem~\ref{thm:non-mono}]
The \lq anchored\rq\ cubic map in $\P^3_{-1}$ form a two-parameter family
$$
f_{\alpha,\beta}(x) = \alpha x^3 + \beta x^2 + (1-\alpha)x - \beta,
$$
where $0 \le \alpha \le 4$ and $|\beta| \le 2\sqrt \alpha - \alpha$.
Thus $f_{\alpha,\beta}(\pm 1) = \pm 1$ and the critical points are
$c_{1,2} = \frac{1}{3\alpha} (-\beta \pm \sqrt{\beta^2-3\alpha(1-\alpha)})$.
It is possible to parametrize the family $\P^b$ by critical values
$v_1,v_2$, see \cite[Theorem II.4.1]{MS}, but the formulas are unpleasant.
For our purposes it suffices to verify that parameter $\alpha$ depends
monotonically on $v_1$ when $v_2$ is kept constant (close to $-1$).

For $\beta = 2\sqrt \alpha - \alpha$, the rightmost critical point maps to $-1$,
and we will use this parameter choice for the moment.
That is, $c_1 = -\frac13(1+\frac{1}{\sqrt{\alpha}})$,  
$c_2 = 1-\frac{1}{\sqrt{\alpha}}$, and these critical points are increasing
in $\alpha$.
The first critical value 
$$
v_1 = f_{\alpha,2\sqrt{\alpha}-\alpha}(c_1) = \frac{32}{27}\alpha - \frac{48}{27} \sqrt{\alpha} - \frac19 - \frac{4}{27 \sqrt{\alpha}}
$$
is increasing in $\alpha$ as well.

In fact, numerics show that $\frac{d}{d\alpha} f_{\alpha,\beta}(x) > 0$ for all 
$1 \le \alpha \le 4$ and $x \in (-1, c_1]$.
It follows that for every two values of $v_1$, say $v_1(\alpha) < v_1(\alpha')$
for $1 \le \alpha < \alpha' \le 4$ and $x \in (0, c_1(\alpha))$, we have
$f_{\alpha,2\sqrt{\alpha}-\alpha}(x) < f_{\alpha',2\sqrt{\alpha'}-\alpha'}(x)$.

It can also be verified that the parameter
$\alpha$ such that $f_{\alpha,2\sqrt{\alpha}-1}(c_1) > c_2 > f^2_{\alpha,2\sqrt{\alpha}-1}(c_1) = c_1$
lies within the interval $[3.668, \ 3.670]$.
For these values of $\alpha$, 
$v_1 = f_{\alpha,2\sqrt{\alpha}-\alpha}(c_1) \approx 0.75$
and the left-most preimage 
$u \in f^{-1}_{\alpha, 2\sqrt{\alpha}-\alpha}(c_2) \approx -0.72$.
Also $c_1$ is still attracted to a period $2$ orbit, and the entropy
$\htop(f_{\alpha,2\sqrt{\alpha}-\alpha}) = \log 1+\sqrt{2}$.
We find numerically that for $x \approx -0.72$,
\begin{equation}\label{eq:dfudc2}
\frac{d}{d\alpha} f_{\alpha, 2\sqrt{\alpha}-\alpha}(x) > 2 > \frac18 > \frac{d}{d\alpha}(c_2(\alpha)) = \frac1{2\alpha\sqrt{\alpha}}.
\end{equation}

Now we fix some large $n$, $v_1^*$ close to $0.75$
and $v_2^*$ close to $-1$ such that
$v_2^* = f_{\alpha,\beta}(c_2) < f^2_{\alpha,\beta}(c_2) ,< \dots < f^{n-1}_{\alpha,\beta}(c_2) < c_1 < f^n_{\alpha,\beta}(c_2) = c_2$.
That is, from now on $\alpha,\beta$ and $c_{1,2}$ depend on $v_1$ and $v_2$, where 
$v_2 = v_2^*$ is fixed and $v_1$ is the remaining variable such that
$f^n_{\alpha,\beta}(c_2) = c_2$ for $v_1 = v_1^*$ and $v_2 = v_2^*$.
We will write $f_{\alpha, \beta} = g_{v_1, v_2}$.
Since $v_2^*$ can be taken arbitrarily close to $-1$ (at the price of 
taking $n$ large),
we can assume that $\alpha$ is still an increasing function of $v_1$,
and the estimates leading up to \eqref{eq:dfudc2}
remain valid.

This means that $\frac{d}{dv_1}(g_{v_1, v_2^*}(x)) > 0$ for every fixed
$x \in (0,c_1)$, and for values $x \approx g_{v_1, v_2^*}^{-1}(c_2)$,
we have $\frac{d}{dv_1}(g_{v_1, v_2^*}(x)) > \frac{d}{dv_1}(c_2)$.
Therefore, $\frac{d}{dv_1} g^k_{v_1,v_2^*}(c_2) > 0$ for $1 < k < n$
and  $\frac{d}{dv_1} g^n_{v_1,v_2^*}(c_2) >\frac{d}{dv_1}(c_2)$.

Again, for $v_1 = v_1^*$, the second critical point $c_2$ is (super)attracting 
of period $n$.
Let us denote by $q(v_1)$ the continuation of this (attracting) $n$-periodic.
If $U$ is a small neighborhood of $c_2$,
$f^{n+1}_{v_1, v_2^*}|_U$ has two branches for $v_1 \ge v_2^*$ and
four branches for $v_1 < v_1^*$.
If $v_1 > v_1^*$ is so large that $q(v_1)$ disappears in a saddle-node 
bifurcation, then $\htop(g_{v_1, v_2^*})$ decreases compared to $\htop(g_{v_1^*,v_2^*})$.
However, for $v_1 \approx 1$, \ie $g_{v_1, v_2^*}$ close to the 
cubic Chebyshev polynomial, the entropy is close to $\log 3$, and hence
$v_1 \mapsto \htop(g_{v_1, v_2^*})$ cannot be monotone.
\end{proof}

\section{A lemma on increasing entropy}
One aspect of increase of entropy is that it always involves the creation
of periodic orbits, see equation \eqref{eq:htop}.
However, not every creation of a periodic orbit increases the entropy;
namely if orbits are created within an existing non-trivial periodic interval, 
then this has no effect on the entropy.
The next results shows that for piecewise monotone maps, if no periodic intervals are present, entropy is strictly monotone in single changes of critical values.

\begin{lemma}\label{lem:strictlymonotone}
For each piecewise monotone map $T\colon [-1,1,\to [-1,1]$ and for each $\kappa>0$ there exists $\xi = \xi(\kappa, T)>0$ with the following properties.
Assume that $f,\tilde f\colon [-1,1] \to [-1,1]$ are piecewise monotone maps (possibly with plateaus), 
$U$ is an interval and let $J_1,\dots,J_k\subset U$ be the maximal 
intervals on which $f$ is strictly monotone or constant. 
Moreover, assume that 
\begin{itemize}[topsep=-0.1cm,itemsep=0ex,leftmargin=0.8cm]
\item $f$ is semi-conjugate to $T$ and the semi-conjugacy only collapses plateaus (and their preimages) and basins of periodic attractors of $f$;
\item 
$\tilde f(J_i) \supset f(J_i)$ for each $i=1,2,\dots,k$ and so that 
$|\tilde f(J_i)|\ge (1+\kappa)|f(J_i)|$ for at least one $i$,
\item $\tilde f(J) \setminus f(J)$ is not eventually mapped into a periodic cycle
of intervals (other than $[-1,1]$ itself), is not attracted to periodic attractor, nor contained in a wandering interval.
\end{itemize}
Then $\htop(\tilde f) > (1+\xi)\htop(f)$.
\end{lemma}

\begin{proof} This is proved in \cite[Lemma 7.2]{BS}. \end{proof}

\section{Isentropes which are non-locally connected}\label{sec:locally_disconnected}

In this section we shall prove the existence of isentropes which 
are non-locally connected in the space $\P^b_\shape$ for any $b\ge 3$
and $\shape=\pm 1$: we will show that there exists a dense set of $h\in [0,\log(b-1)]$ for which $L_{h}\subset \P^b_\shape$ is non-locally connected.  

\begin{definition}\label{def:fundamental_domain}
Given an interval map $g$, we say that an interval $F$ is a {\em fundamental domain} in the basin of an attracting orbit of period $N$ if
\begin{enumerate}
 \item[(i)] $F$ and $g^N(F)$ have precisely one point in common, and 
\item[(ii)] $\partial g^N(F)=g^N(\partial F)$.
\end{enumerate} 
This notion agrees with the usual notion in the monotone case, 
but the assumption (ii) is added to deal with the case that
$g^N_t|J_t$ is not monotone.
\end{definition}

\begin{theorem}\label{thm:nonloc-sn}
Assume that $s = e^h \in [0, b-1]$ and let 
$T \in \T^{b-2}_\shape$ be a map with slope $\pm s$ with $b-2$ periodic turning points. 
Let $\hat q$ be an $N$-periodic point of $T$ so that $\cup_{n\ge 0}T^{-n}(\hat q)$ is dense in $[-1,1]$
and so that the orbit of $\hat q$ does not contain a turning point of $T$.
Then there exists $\delta>0$ such that if
 $L_h\subset \P^b_\shape$ contains a map $g$ with the properties:
\begin{enumerate}[topsep=-0.15cm,itemsep=0.3ex,leftmargin=0.6cm]
\item $g$ has a periodic point $q^*$ of saddle-node type of period $N$;
\item two adjacent critical points $c_\eta,c_{\eta+1}$ are both contained in a fundamental domain 
$F$ in the immediate basin of $q^*$ and  $|c_\eta-c_{\eta+1}|<\delta |F|$;
\item $g$ is semi-conjugate to $T$  where the semi-conjugacy only
collapses basins of periodic attractors.
\item each other critical point of $g$  is in the basin of a hyperbolic periodic attractor;
\item arbitrarily close to $g$, there exist maps $\tilde g$ in $\P^b_\shape$ which do not have a parabolic periodic point of period $\le N$,
\end{enumerate}
then $L_h$ is not locally connected at $g$.
\end{theorem}

We note that since $c_\eta,c_{\eta+1}$ are the only two critical points in the 
basin of $q^*$, 
there exists a fixed point $q$ of $g^N$ so that the component of the 
basin of $q^*$ 
containing $q^*$ is equal to  $(q,q^*]$, see Figure~\ref{fig:saddle-node}. 

The topological entropy of a map $g\in \P^b_\shape$ as in this theorem, 
is the same as that of  polynomial map $\hat g$  for which $c_\eta=c_{\eta+1}$ and which therefore has a 
degenerate critical point at $c_\eta=c_{\eta+1}$;  such a map $\hat g$ has $b-2$ turning points and therefore 
at most topological entropy $\log(b-1)$. Since this is the only
mechanism which we are aware of for obtaining non-local connectivity of isentropes, this explains the reason that we have to assume that $h\in [0,\log(b-1)]$
and can only  assert that there are non-locally connected isentropes in $\P^b_\shape$ when $b\ge 3$.

{\bf Remark:} 
The assumption that $\cup_{n \ge 0} T^{-n}(\hat q)$ is dense in $[-1,1]$
implies that if $T$ is renormalizable, say $U \supset T^m(U)$ for some non-degenerate interval $U$,
then $\hat q \in \orb(U)$. If the maximal such $m = p \cdot 2^r$ for some odd
integer $p \ge 3$, then 
$\htop(T) \ge 2^{-r} \log \lambda_p$ where $\lambda_p$
is the largest solution of $x^p - 2x^{p-2} = 1$ (see \cite[Theorem 4.4.17]{ALM}). Hence to realize small values of $h$, the integers $N$ and $k$ in condition C2 below must be multiples of $2^r$ for some large $r$.

Since the proof for $b=3$ (with $c_\eta = c_1$) is the same as the general case, we present here only that case.
For simplicity we shall also assume that $\shape=-1$.  
 
\subsection{Construction of a map \boldmath $g$ \unboldmath satisfying the assumptions of Theorem~\ref{thm:nonloc-sn}}
\label{subsec:construction}

\begin{figure}
\begin{center}
\unitlength=1.3mm
\begin{picture}(20,15)(10,5)
\put(-30,0){\resizebox{4cm}{4cm}{\includegraphics{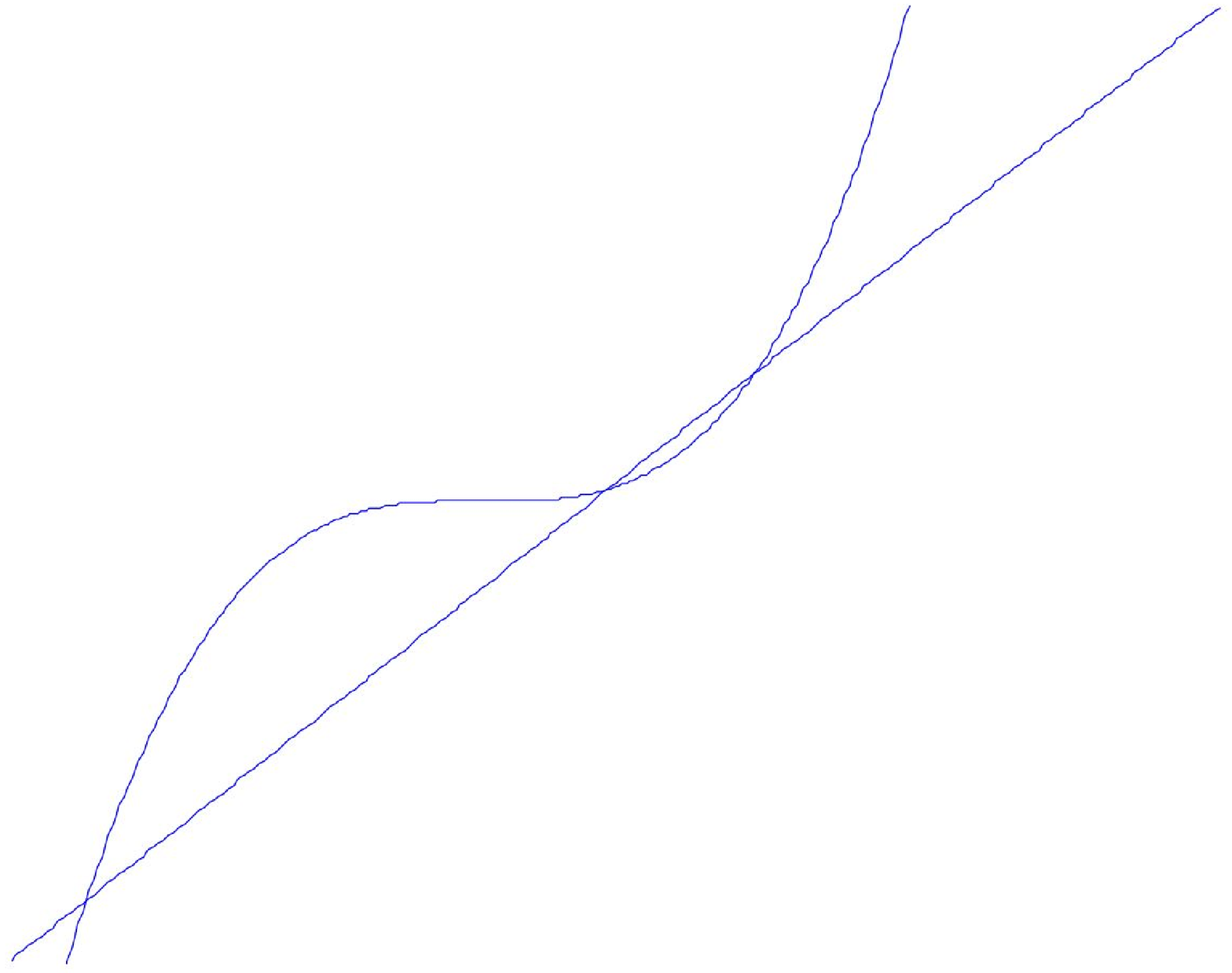}}}
\put(0,0){\resizebox{4cm}{4cm}{\includegraphics{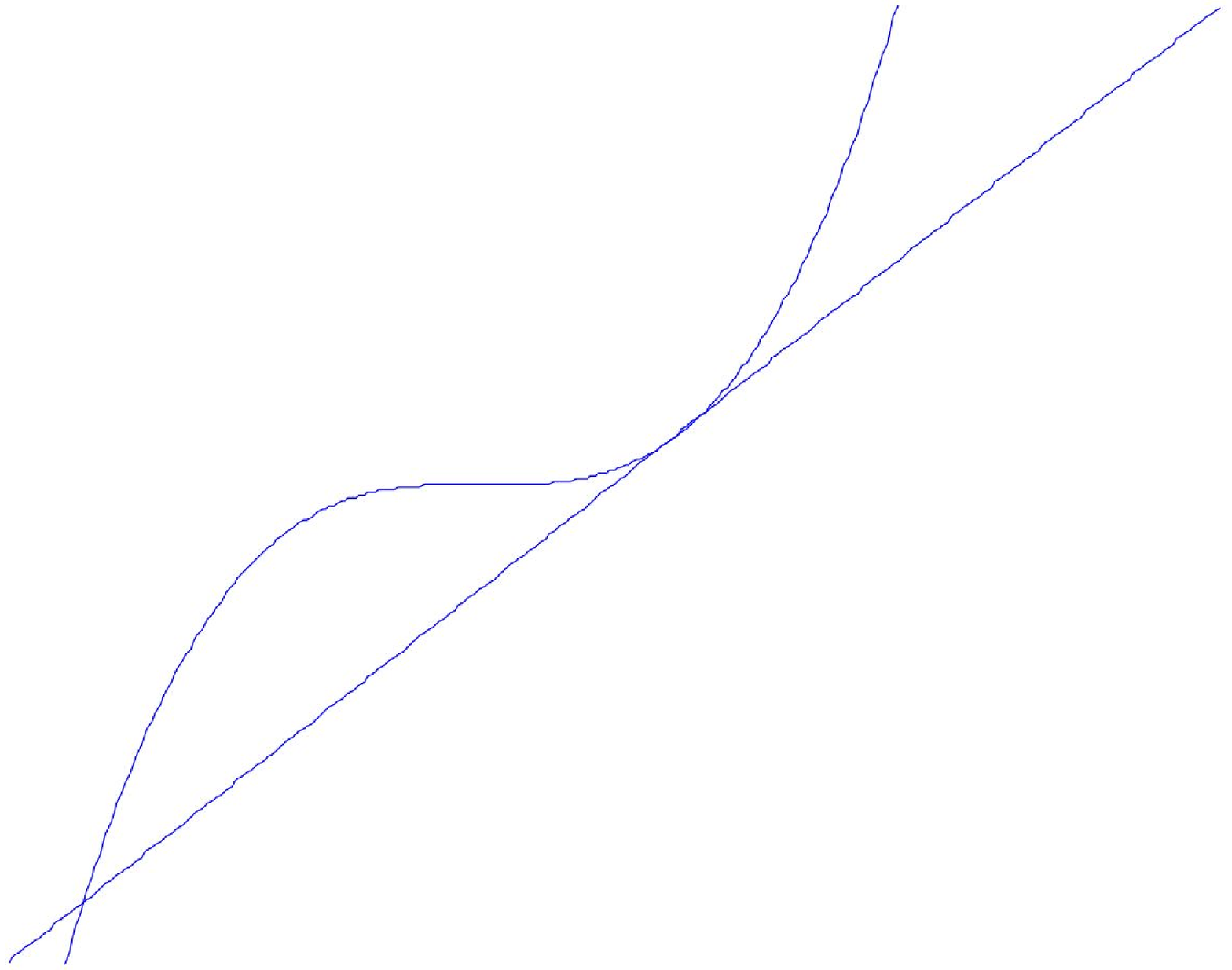}}}
\put(30,0){\resizebox{4cm}{4cm}{\includegraphics{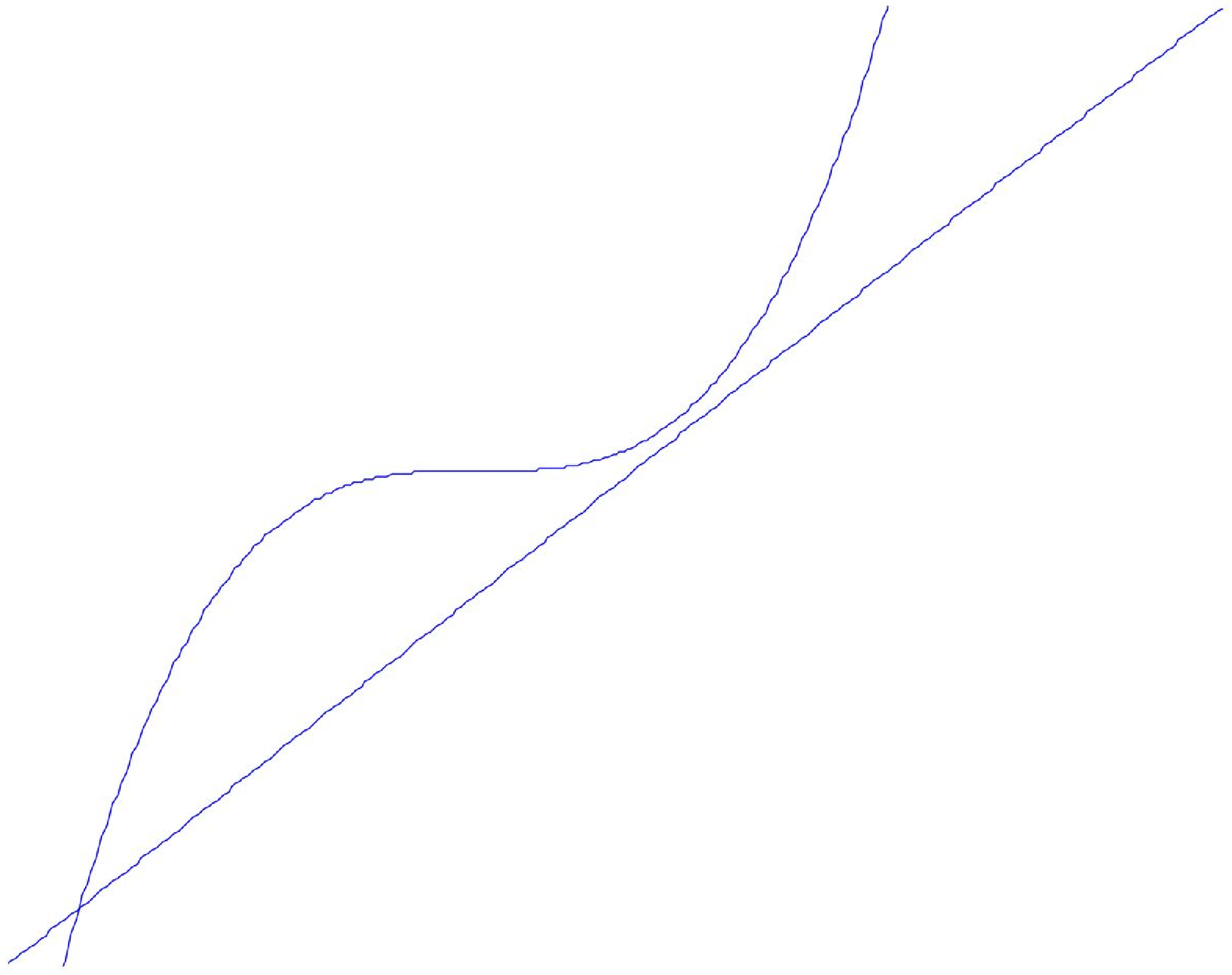}}}
\end{picture}
\end{center}
\caption{\label{fig:saddle-node} The graph of $\tilde g^N$ for $\tilde g$ near $g$. 
In the middle figure,  $g^N$ has a hyperbolic fixed point $q$ and a
saddle-node fixed point $q^*$.}
\end{figure}

To make the discussion rather explicit we will first show that maps satisfying 
the assumptions in Theorem~\ref{thm:nonloc-sn} exist.
To do this, choose a unimodal tent map 
(in the general case we choose a map with constant slope with  $b-2$ turning points and replace $\log 2$ below by $\log(b-1)$). 
 Let $h\in (0,\log 2)$ be so that there exist an integer $R$, $s>1$
and a map $T \in \T^3_\shape$ with slopes $\pm s$ (where $s=\exp(h)$) and 
turning point $c=0$ for which  $T^R(c)=c$. 
Note that the set of such $h$'s is dense in $[0,\log 2]$. 
For later use, pick any periodic point $\hat q$ of $T$ so that $\cup_{n\ge 0}T^{-1}(\hat q)$ is dense
and so that the orbit of $\hat q$ does not contain a turning point of $T$.
Let $N$ be its period. Given $\hat q$,  choose an integer $k$, an interval neighbourhood 
$\hat J \owns \hat q$ and  a point $\hat x$, so that  
\begin{enumerate}[topsep=-0.1cm,itemsep=-0.1ex,leftmargin=0.6cm]
\item[C1] $T^N|\hat J$ is monotone,
\item[C2] $\hat x\in T^N(\hat J)$ and there exists a neighbourhood $\hat V\owns \hat x$ so that $T^k$ 
maps $\hat V$ monotonically onto $\hat J$. 
\end{enumerate} Since $\cup_{n\ge 0}T^{-1}(\hat q)$ is dense, 
$k$, $\hat J$ and $\hat x$ always exist. 
We may assume that $T^N|\hat J$ is orientation preserving (otherwise replace $N$ by $2N$)
and in order to be definite we assume that  $\hat q$ lies to the left of the turning point
of $T$ and that $\hat x>\hat q$  (if $\hat x<\hat q$ then we replace $\sup$ by $\inf$ in the proof below).

Consider the space $\overline \P^b_\shape$ of real degree four polynomials 
with three critical points $-1\le c_1\le c_2<c_3<1$  and so that 
$f(\pm 1)=\shape=-1$. Note that $\overline \P^b_\shape$ can be parametrized by
the {\em critical values} $v_i=f(c_i)$ of $f$ and so corresponds to the space
$$
\{(v_1,v_2,v_3)\in \R^3; -1\le v_2\le  v_1\le 1, v_1>-1\mbox{ and }v_2<v_3\le 1\},
$$
see \cite[Section II.4]{MS} and also \cite{CS}.
Note that $\P^b_\shape$ consists of the polynomials in $\overline \P^b_\shape$ for which $v_2<v_1$ holds; by {\em singular} polynomials we mean those in
$\overline \P^b_\shape\setminus \P^b_\shape$.
The next three steps involve singular maps; in Step 4., we construct 
the required non-singular map from them.

\noindent
{\bf Step 1. Construction of a family of singular maps with the appropriate topological entropy.}
For any singular map $f \in \overline \P^b_\shape\setminus \P^b_\shape$ 
one has $c_1=c_2$ and therefore $f$ is unimodal. 
Since $v_3$ can take any value in $(-1,1]$ this forms a full family of 
unimodal maps
and therefore we can  pick $g_0\in\overline \P^b_\shape\setminus \P^b_\shape$ 
so that $g_0$ is semi-conjugate to $T$
and so that $g_0^R(c_3(g_0))=c_3(g_0)$. Since the semi-conjugacy only collapses
the basin of periodic attractor of $g_0$, we have $\htop(g_0)=h$.
Next define $\hat \Gamma = \{ g\in \overline \P^b_\shape\setminus \P^b_\shape \ : \ g^R(c_3(g))=c_3(g)$ and $\Gamma$ the component of $\hat \Gamma$ which 
contains $g_0$.
Note that $\Gamma$ is a real algebraic variety of  (real) dimension one. Moreover, by transversality, 
see \cite{epstein} and also  \cite{vS},  $\Gamma$ is a one-dimensional manifold.
As $g$ varies  in $\Gamma$ the critical points $c_1(g)=c_2(g)$ 
move from the left endpoint of $[-1,1]$ to $c_3(g)$, see also \cite{CS} . More precisely, 
one can take a one-parameter family of maps $f_t$, $t\in (0,1)$ parametrizing $\Gamma$,
so that as $t\downarrow 0$, $c_1(t)=c_2(t)\to -1$ and as $t\uparrow 1$ 
the distance of $c_1(t)=c_2(t)$ to $c_3(t)$ tends to zero.
For each $t\in (0,1)$ the map $f_t$ is semi-conjugate to $T$ by a
semi-conjugacy $H_t$  which only collapses the basin of the periodic attractors of $f_t$.
(In order to be definite we choose normalisations  so that $x\mapsto H_t(x)$ is increasing). 
Hence
$$
f\in \Gamma \implies \htop(f)=h.
$$

\noindent
{\bf Step 2. Construction of a singular map  \boldmath $f_{t'}$  within this family with a periodic point of saddle-node type
containing $c_1(t')=c_2(t')$  \unboldmath in its immediate basin.}
For each $t\in (0,1)$, the image $H_t(c_1(t))$ of the inflection point $c_1(t)=c_2(t)$
lies to the left of the turning point of $T$. Moreover, $t\mapsto H_t(c_1(t))$ is continuous
because the kneading invariant of $c_1(t)$ with respect to the partition $[-1,c_3(t)]$ and $[c_3(t),1]$
depends continuously on $t$ except when $c_1(t)$ is eventually mapped to $c_3(t)$.
However, even in that case $t\mapsto H_t(c_1(t))$ is continuous since $c_3(t)$ is a periodic 
attractor and hence a neighbourhood of $c_3(t)$ is mapped to the turning point of $T$. 

As mentioned above, as $t\in (0,1)$ varies from $0$ to $1$, 
we have that $c_1(t)=c_2(t)$ varies from the left endpoint of $[-1,1]$ to $c_3(t)$.
In particular, by continuity of $t\mapsto H_t(c_1(t))$, and since $\hat q$ 
is to the left of the turning point of $T$,   the set $A$ of parameters $t$ so that 
$H_{t}(c_1({t})) = \hat q$ is the $N$-periodic point of $T$ from above, 
is non-empty. Since $f_t$ is semi-conjugate
to $T$, for each $t\in A$, there exists an interval $J_t \subset f^N_t(J_t)$ 
so that $f_t^N|J_t$ is monotone and orientation preserving and 
$H_t(J_t)=\hat J$ is the interval associated to $T$ from above.
Let $U_t$ be the maximal interval in $J_t$ so that $H_t(U_t)=\hat q$. 
Since $T^N(\hat q)=\hat q$ and $f_t$ is semi-conjugate to $T$ we get
$f_t^N(U_t)=U_t$. Since $\hat q$ is contained in the interior of $\hat J$, 
the interval $U_t$ is compactly contained in $J_t$. Moreover, 
since $c_1(t)=c_2(t)\in U_t$, the interval $U_t$ is non-degenerate. 
Therefore and since $f_t^N|U_t$ is monotone, both endpoints of 
$U_t$ are fixed points under $f^N_t$ and $c_1(t)=c_2(t)$ are in the basin
of a periodic attractor within $U_t$. Define $t'=\sup A$. 
This means that $f_{t'}$ has an $N$-periodic parabolic point $q^*$ so that
$c_1(t')=c_2(t')$ are in its basin. 
(Otherwise $c_1(t')=c_2(t')$ would be in the basin 
of a hyperbolic periodic attractor of period $N$ contradicting that $t'=\sup A$.)
Let $q(t'),q^*(t')$ be the boundary points of $U_{t'}$, where
$q^*(t')$ is the parabolic point. The other, $q(t')$, is hyperbolic repelling,
since $f_{t'}$ has negative Schwarzian derivative, and both have period $N$.
As $t$ varies from $0$ to $1$, \,\, $H_t(c_1(t))$ varies from $-1$ to the turning point of $T$.
That $t'=\sup A$ implies that $H_t(c_1(t))>\hat q$ for $t>t'$.  Since $H_{t'}(q(t'))=\hat q$
it follows that  $q(t')<c_1(t')=c_2(t')$. Since  $f_t^N|J_t$ is monotone increasing
we therefore get $q(t')<c_1(t')=c_2(t')<q^*(t')$.

Furthermore, again because $f_{t'}$ is semi-conjugate to $T$,
there exists a point $x(t')\in f_{t'}^N(J_{t'})$ so that $H_{t'}(x(t'))=\hat x$ and so that there exists an interval
$V_{t'}\ni x(t')$ so that  $f_{t'}^k$ maps $V_{t'}$ diffeomorphically onto $J_{t'}$, 
where $k$ and $\hat x$ are associated to $T$ as above. Since $t'=\sup A$,
\begin{enumerate}[topsep=-0.1cm,itemsep=0.3ex,leftmargin=0.6cm]
\item  for each $t\in [t',1)$ the map  $f^N_{t}|J_t$ has a hyperbolic repelling fixed point $q(t)\in J_t$; 
\item  for any $t\in (t', 1)$ and any $y\in J_t$ with $y>q(t)$ one has $f^N_t(y)>y$;
\item $f^N_{t'}$ has a parabolic fixed point $q^*(t')\in J_{t'}$;
\item $J_t$ contains $q(t')<c_1(t')=c_2(t')<q^*(t')$ for each $t$ near $t'$.
\end{enumerate}
This shows that $f_{t'}$ has a saddle-node periodic point of the required type. 

\noindent 
{\bf Step 3. Construction of a sequence of singular maps 
 \boldmath $f_{t_n}$ with $t_{n}\downarrow t'$
and for which $c_1(t_n)=c_2(t_n)$  \unboldmath
are attracting periodic orbits.}
Since $T^k\colon \hat V\to \hat J$ is monotone, there exists a sequence
of periodic points $\hat x_n\in \hat J$, so that
$$\hat x_n,T^N(\hat x_n),\dots,T^{(n-1)N}(\hat x_n)\in \hat J\mbox{ , }
T^{nN}(\hat x_n)\in \hat V\mbox { and }T^{nN+k}(\hat x_n)=\hat x_n.$$  Of course $\hat x_n\downarrow \hat x$
as $n\to \infty$. 
It follows that there exists $t_n\in (t',1)$ so that $H_{t_n}(c_1(t_n))=H_{t_n}(c_2(t_n))=\hat x_n$ and so that
$f_{t_n}^{nN+k}(c_1(t_n))=c_1(t_n)$. 

We claim that $t_n\downarrow t'$. Indeed,  $n$ iterates of $c_1(t_n)$
under the monotone map $f_{t_n}^N|J_{t_n}$ stay in $J_{t_n}$. 
It follows that there exist $0<k_n<n$ so that $|f_{t_n}^{k_nN}(c_1(t_n))-f_{t_n}^{(k_n+1)N}(c_1(t_n))|\to 0$
as $n\to \infty$. Since $q(t)$ is a repelling $N$-periodic point for each $t\in (t',1)$, 
the critical point $c_1(t_n)$ cannot converge to $q(t)$ as $n\to \infty$. 
Therefore $|f_{t_n}^{k_nN}(c_1(t_n))-f_{t_n}^{(k_n+1)N}(c_1(t_n))|\to 0$ implies that
$f_{t_n}^{N}|J_{t_n}$ has an {\lq}almost fixed point{\rq} when $n$ is large. 
Since $f^N_{t}|J_t$ has only one fixed point as $t>t'$ (namely $q(t)$), it follows that $t_n\to t'$.

Again there exists a sequence of parameter intervals $A_n\ni t_n$
so that $c_1(t_n)=c_2(t_n)$ belongs to the basin of an periodic attractor of period $nN+k$.
The parameter intervals $A_n=[t_n^-,t_n^+]$ are disjoint and converge to $t'$.
Here $t_n^\pm$ are parameters at which $f^{nN+k}_{t_m^\pm}$ has a parabolic fixed point
(this holds because for each $t\in [t_n^-,t_n^+]$, the map $f^{nN+k}_t$ is monotone on the component of the basin
of the periodic attractor containing $c_1(t_n)=c_2(t_n)$).

\noindent 
{\bf Step 4. Construction of a non-singular map which satisfies the assumptions
of Theorem~\ref{thm:nonloc-sn}.}
The singular map $f_{t'}$ has a parabolic $N$-periodic point and there exists a codimension one algebraic
set $\Sigma(t') \ni f_{t'}$ of maps in $\overline \P^b_\shape$ which also have a parabolic periodic point
of period $N$. Similarly, there exist codimension one algebraic sets $\Sigma(t_n^\pm)$ 
containing $f_{t_n^\pm}$ with a parabolic periodic point of period $nN+k$.  
Using the rigidity results contained in \cite{KSS} and \cite{BS}, one can prove  that 
$\Sigma(t')$ and $\Sigma(t_n^\pm)$ are codimension-one manifolds of $\P^b_\shape$ and that $\Sigma(t_n^\pm)$
converges to $\Sigma(t')$. So for maps corresponding to the region between $\Sigma(t_n^-)$ and $\Sigma(t_n^+)$ 
the critical points $c_1,c_2$ are both contained in the same component 
of the immediate basin of a periodic point of period $nN+k$.
So the situation is somewhat analogous to that of rational Arnol'd tongues for circle diffeomorphism (except that 
the regions between   $\Sigma(t_n^-)$ and $\Sigma(t_n^+)$ do not touch). 
The map $g$ we are looking for lies in the set $\Sigma(t')$. 
Rather than proving that $\Sigma(t')$ is a manifold in this paper,
to show that the required $g$ exists,  it will be sufficient to use much softer arguments. 

Indeed, take $t^-<t'<t^+$ so that $f^N_{t^-}$ has a hyperbolic attracting fixed near $q^*(t')$ and 
$f^N_{t^+}$ has no fixed point near $q^*(t')$. Moreover, let $g_t\in \P^b_\shape$, $t\in [t^-,t^+]$ be a family 
which depends continuously on $t$ and so that $d(f_t,g_t)<\delta'$ for each $t\in [t^-,t^+]$
(where $d(f,g)$ is the Euclidean distance of the coefficients of the two polynomials).
Provided $\delta'>0$ is sufficiently small, there exists $t_p\in [t^-,t^+]$ so that 
$g:=g_{t_p}$ has a parabolic $N$-periodic point, so that
$g_{t}$ has no parabolic periodic point of period $\le N$ for $t\in (t^-,t^+]$ 
and so that $g$ satisfies the properties
required in the assumption of Theorem~\ref{thm:nonloc-sn}.

\subsection{Proof of Theorems~\ref{thm:nonloc-sn} and \ref{thm:isentropeNLC}}

Before proving  Theorems~\ref{thm:nonloc-sn} and \ref{thm:isentropeNLC}, we 
state and prove some lemmas.

Let $f_{t_n}$ be the singular maps constructed in Step 3.\ above. 
Take $f_n\in \P^b_\shape$ sufficiently close to $f_{t_n}$ so that $f_n$ has a periodic attractor of
$nN+k$ which still attracts the two critical points $c_1$ and $c_2$ and so that, moreover, $c_3$ is still in the basin
of a periodic attractor of period $R$. 

Set
\begin{equation}\label{eq:Wn}
W_n := \{ \varphi \in  \P^b_\shape \ : \
\varphi \text{ is hyperbolic and partially conjugate to } f_n\},
\end{equation} 
where we say that 
$\varphi$ and $f_n$ are {\em partially conjugate} if there exists an 
orientation preserving homeomorphism $h$ 
(depending on $\varphi$ and $f_n$) so that $h\circ f_n=\varphi\circ h$ 
holds outside the basin the periodic attractors of $f_n$.

\begin{lemma}
The set $W_n$ is homeomorphic to an open three-dimensional ball. 
\end{lemma}

\begin{proof}
This lemma is special case of what is proved in \cite[Theorem 2.2]{BS}. 
Note that  the sets $W_n$ are pairwise disjoint since each these sets correspond to maps with periodic attractors
of different periods, namely $nN+k$.
\end{proof}

As mentioned above, the boundary of $W_n$ consists of smooth codimension one surfaces in $\P^b_\shape$, but we shall not use or prove this fact here. 

Let $n$ be any integer.  In the next lemma we will show that entropy is oscillating along any path of maps
 $g_t$ which connects a map in $W_n$ (from \eqref{eq:Wn}) with $g$
 where $g$ is the map from Theorem~\ref{thm:nonloc-sn}.

\begin{lemma}\label{lem:tongues}
Take $h,T$ and $N$ as in the assumptions of Theorem~\ref{thm:nonloc-sn}.
Then there exists $\delta_0>0$, so that for any
$g\in \P^b_\shape$ which satisfies conditions (1)-(5) from that theorem with $\delta\in (0,\delta_0)$, there exist $h_0>h_1>h_2>\dots>h$ and $\tau_0>0$ 
so that
for {\em any} unfolding family $g_t$ for which $d(g_t,g)<\tau_0$ for each $t\in [0,1]$
and  with $g_{t_n}\in W_n$ for some $t_n\in (0,1]$ there exist parameters  $t_n>s_n>t_{n+1}>s_{n+1}> \dots >0$ 
so that for each $m\ge n$ the following holds:
\begin{itemize}[topsep=-0.1cm,itemsep=0ex,leftmargin=0.8cm]
\item[(A)]  $g_{t_m}\in W_m$  and therefore $\htop(g_{t_m})=h$;
\item[(B)]  $\htop(g_{s_m})>h_m>h$.
\end{itemize}  
\end{lemma}

{\bf Remark:} Provided $\tau_0>0$ is sufficiently small, $g_t\in \P^b_\shape$
is non-singular for each $t\in [0,1]$.

\begin{proof} Let us again assume that $b=3$.  For $\delta>0$ sufficiently small, 
all critical points $c_3,\dots,c_b$ are still in the basins of 
hyperbolic periodic attractors. 
Define  $\hat c_1(t)\in (c_2(t),c_3(t))$ so that $g_t(\hat c_1(t))=g_t(c_1(t))$
and define 
$$
G_t(x) = \begin{cases} 
g_t(c_1) & \text{ if  } x \in [c_1(t),\hat c_1(t)]; \\
g_t(x) & \text{ otherwise.  }
\end{cases}
$$
Since the kneading invariants of $G_t$ and
$T$ for the critical points $c_3,\dots,c_b$ are the same, there exists a 
semi-conjugacy $H_t$ between $G_t$ and $T$. 
Let $\hat q=H_t(q^*)$ and let $k$, $\hat J$ and $\hat x$ be as in Step 1 
in Section~\ref{subsec:construction}. Moreover, 
choose an interval $J_t\ni q^*$ so that $H_t(J_t)=\hat J$ and take $x_t\in g^N_t(J_t)$ so that $H_t(x_t)=\hat x$.
As in Step 3 of Section~\ref{subsec:construction}, 
there exists parameters $t_n>t_{n+1}>\dots >0$ so that 
for $m\ge n$ one has $g_{t_m}^{mN+k}(c_1(g_{t_m}))=G_{t_m}^{mN+k}(c_1(g_{t_m}))=c_1(g_{t_m})$. 
Provided $\delta>0$ is sufficiently  small, 
for each $t\in [0,1]$ there exists a fundamental domain
(recall Definition~\ref{def:fundamental_domain})
 $F_t\subset J_t$  containing $c_1(g_t)$ and $c_2(g_t)$.
Now take $t>0$ and $n$ so that  $F_t,\dots,g_t^{(n-1)N}(F_t) \subset J_t$ 
have disjoint interiors and so that  $x_t\in g_t^{nN}(F_t)$.  

{\bf Claim:} There exists $\rho_0>0$ so that $|g_t^{nN}(F_t)|, |F_t|\ge \rho_0$ for each $t>0$. 

Indeed, let $U_0$ be the maximal neighbourhood of $c_1(t)$
on which $|Dg_t(x)| < 1$, and let $U_1 \owns c_1(t)$ be a subinterval 
of diameter $\frac14 |U_0|$.
If $F_t \subset U_1$, then since also $g_t^N(F_t)$ is adjacent to $F_t$
(recall that it is a fundamental domain)
$c_1(t)$ is contained in the basin of an attracting fixed point of 
$g_t^N|J_t$, contradicting assumption (c) on the family $g_t$. 
That $|g_t^{nN}(F_t)|$ is not small holds because this interval is about to ``leave $J_t$'' (in less than $k$ steps, where $k$ does not depend on $n$). 

The previous claim implies that  $Dg_t^{mN}|F_t$ is bounded uniformly in $t>0$ for each $m\ge n$.
Since $k$ does not depend on $t$, it follows that  $Dg_{t_n}^{mN+k}|F_t$ is also 
bounded uniformly in $t>0$. In particular, 
when $\delta>0$ is sufficiently small, 
$|Dg_{t_m}^{mN+k}|<1$ on the interval connecting $c_1(g_t)$ and $c_2(g_t)$
and therefore  $c_2(g_t)$ is contained in the basin of attracting periodic orbit $c_1(g_t)$ when $t=t_m$.
This shows that $g_{t_m}\in W_m$ and therefore $\htop(g_{t_m})=h$.
In other words, the family $f_t$ passes through $W_m$, which means that 
the {\lq}teeth of the comb{\rq} in the $h$-isentrope don't decrease in height as $m \to \infty$.  

There exists a sequence $s_n>s_{n+1}>\dots>0$ so that $t_n>s_n>t_{n+1}>s_{n+1}>\dots>0$ and so that $c_1(g_{s_m})$
is not in the basin of a periodic attractor of  $g_{s_m}$ for $m\ge n$. This holds because the set of parameters $s$
for which  $c_1(g_s)$ is contained 
in the basin of a periodic attractor of $g_s$ is a union of disjoint intervals with disjoint closures, and therefore the complement has to be non-empty.
But as $[c_1(s_m), c_2(s_m)]$ is not asymptotic to a periodic orbit, it must eventually be mapped onto an entire interval component of the `deepest' renormalization cycle of $g_{s_m}$. 
It follows from Lemma~\ref{lem:strictlymonotone} that 
$\htop(g_{s_n})>h_n>h=\htop(G_{s_n})$
and that moreover $h_n$ only depends on $f$ and $\tau_0$.
\end{proof}

\begin{lemma}\label{lem:conn}
Let $X$ be connected subset of $\R^b$. 
Then for each $\tau>0$ there exists an open set $U$ with $X\subset U \subset N_\tau(X)=\{y \ : \ d(y,X)<\tau\}$ which is path connected.
\end{lemma}
\begin{proof} For each $x\in X$, choose $\tau(x)>0$ and 
let $U=\cup_{x\in X}B_{\tau(x)}(x)$. Since $X$ is connected, so is $U$. 
Since $U$ is open, it is also path connected.
\end{proof}

Now we are ready to prove the main theorems.

\begin{proof}[Proof of  Theorems~\ref{thm:nonloc-sn}]
Take $g$ and $\tau_0$ be as in the assumption of Theorem~\ref{thm:nonloc-sn}.
Suppose by contradiction that $L_h$ is locally connected at $g$.  Then, by definition,  
there exists an open set $Y$ which is contained in a $\tau_0/2$-neighbourhood of $g$
so that $X:=L_h\cap Y$ is connected.
 By Lemma~\ref{lem:tongues} for $n$ sufficiently large, 
$W_n\cap X\ne \emptyset$.  Pick  such an integer  $n$ and let $h_n>h$
be as in Theorem~\ref{thm:nonloc-sn}. 
By Lemma~\ref{lem:conn}, there exists a path connected open set $U\supset X$
which is contained in $\{\tilde g; \htop(\tilde g)<h_n\mbox{ and }d(\tilde g,X)<\tau_0/2\}$.
So choose a path $(g_t)_{t\in [0,1]}$ in $U$ so that $g_0=g$ and $g_1\in W_n\cap X$.
In particular,  $\htop(g_t)<h_n$ and $d(g_t,g)<\tau_0$ for each $t\in [0,1]$.
However, by Lemma~\ref{lem:tongues}, there exists  $s_n\in (0,t_n)$ (where we can take $t_n=1$)
so that $\htop(g_{s_{n}})>h_{n}$. This contradicts that $\htop(g_t)<h_{n}$ for each $t\in [0,1]$.
Thus we have shown that $L_h$ is not locally connected at $g$.
\end{proof}

\begin{proof}[Proof of Theorem~\ref{thm:isentropeNLC}]
This immediately follows from Theorem~\ref{thm:nonloc-sn}.
\end{proof}

\medskip
\noindent
Faculty of Mathematics,  University of Vienna,\\  
Nordbergstra{\ss}e 15, 1090 Vienna, Austria.\\
\texttt{henk.bruin@univie.ac.at}\\
\texttt{http://www.mat.univie.ac.at/$\sim$\,bruin}

\noindent
Department of Mathematics, Imperial College,\\
180 Queen's Gate, London SW7 2AZ, UK.\\
\texttt{s.van-strien@imperial.ac.uk}\\
\texttt{http://www2.imperial.ac.uk/$\sim$\,svanstri/}

\end{document}